\documentclass[12pt,reqno]{amsart}

\usepackage{latexsym}
\usepackage{amssymb}
\usepackage{graphics}
\usepackage{graphicx}
\usepackage{color}

\renewcommand{\epsilon}{\varepsilon}
\newcommand{\dist}{{\operatorname{dist}}}

\newcommand{\kahler}{K\"ahler }

\newcommand{\PP}{{\mathbb P}}

\newcommand{\R}{{\mathbb R}}
\newcommand{\C}{{\mathbb C}}

\newcommand{\CP}{\C\PP}

\newcommand{\dbar}{\bar\partial}
\newcommand{\ddbar}{\partial\dbar}

\newcommand{\E}{{\mathbf E}}

\newcommand{\Zb}{{\mathbf Z}}

\renewcommand{\phi}{\varphi}

\newcommand{\hcal}{\mathcal{H}}

\newcommand{\ocal}{\mathcal{O}}

\newtheorem{theo}{{Theorem}}[section]
\newtheorem{cor}[theo]{{Corollary}}
\newtheorem{lem}[theo]{{Lemma}}
\newtheorem{prop}[theo]{{Proposition}}

\newenvironment{rem}{\medskip\noindent{\it Remark:\/} }{\medskip}

\title[Logarithmic Bergman kernel]{Logarithmic Bergman kernel and Conditional expectation of Gaussian holomorphic fields }

\author{Jingzhou Sun }
\address{Department of Mathematics, Shantou University, Shantou City, Guangdong
	515063,China} \email{jzsun@stu.edu.cn}
\thanks{The author is partially supported by NNSF of China no.11701353 and the STU Scientific Research Foundation for Talents no.130/760181.}

\date{\today}

\begin{document}

	\begin{abstract}
		We prove the asymptotic of the logarithmic Bergman kernel. And as an application, we calculate the conditional expectation of density of zeros of Gaussian random sections of powers of a positive line bundle that vanish along a fixed smooth subvariety.
	\end{abstract}

	\maketitle
	
	\tableofcontents
	\section{Introduction}
	
	Let $(M,L)$ be a polarized \kahler manifold of dimension $m$. We endow $L$ with a Hermitian metric $h$ with positive curvature. And we use $\omega=\frac{i}{2}\Theta_h$ as the \kahler form. By abuse of notation, we still use $h$ to denote the induced metric on the $k$-th power $L^k$. Then we have a Hermitian inner product on $H^0(M,L^k)$, defined by
	$$<s_1,s_2>=\int_M h(s_1,s_2)\frac{\omega^m}{m!}$$
	 Let $\{s_i\}$ be an orthonormal basis of $H^0(M,L^k)$. Then the on-diagonal Bergman kernel $$\rho_k(z)=\sum |s_i(z)|_h^2$$ has very nice asymptotic expansion by
	the results of Tian, Zelditch, Lu, etc. \cite{Tian1990On, Zelditch2000Szego, Lu2000On, Catlin, MM}. 
	Recall that
	$$\rho_k(z)=\frac{k^m}{\pi^m}[1+\frac{S(z)}{2k}+O(\frac{1}{k^2})]$$
	where $S(z)$ is the scalar curvature of the Riemannian metric associated to $\omega$, and the other coefficients are all functions of the derivatives of the Riemannian curvature tensor.
	Let $V$ be a subvariety(subscheme) of $M$, we denote by $\hcal_{k,V}$ be the subspace of $H^0(M,L^k)$  consisting of sections that vanish along $V$. We will call the Bergman kernel of $\hcal_{k,V}$ the k-th logarithmic Bergman kernel of $(M,L,V)$, denoted by $\rho_{k,V}$.
	
	\

	The asymptotics of on- and off-diagonal Bergman kernel have been extensively used in the  value distribution theory of sections of line bundles by Shiffman-Zelditch and others\cite{bsz0,bsz1,bsz2,bsz3,sz1,sz2,sz3,sz4,sz6,sz7,dsz1,dsz2,dsz3,FengConditional,Feng2019}. Of closest relation to this article are \cite{bsz0,bsz1,bsz2,bsz3,FengConditional}. In \cite{bsz0,bsz1,bsz2,bsz3}, the n-point correlation functions of Gaussian random holomorphic sections of $H^0(M,L^k)$ were calculated, together with the scaling limits of these correlation functions, which very interestingly have universality among manifolds of the same dimension. In particular, in the case of Riemann surfaces, they showed that the scaling limit of pair correlation depends only on the distance of the pair of points. In \cite{FengConditional}, in the case of Riemann surfaces, Feng calculated the conditional expectation of density of critical points given a fixed zero and the conditional expectation of zeros given a fixed critical point of Gaussian random holomorphic sections of $H^0(M,L^k)$. He also calculated the rescaling limits, which also exhibit universality. In particular, the universal scaling limit of conditional expectation of density of zeros with a fixed critical point depend also only on the distance between a point and the given critical point.  It was first shown in \cite{Shiffman1999} by Shiffman and Zeldtich that the mean of the zero currents $[Z_s]$ for $s\in H^0(M,L)$ endowed with the Gaussian measure is the pull back of the Fubini-Study form $\omega_{FS}$, whose difference from $\omega$ is just the $i\ddbar\log $ of the Bergman kernel. So it is not surprising that the conditional expectation of zeros is very closely related to the logarithmic Bergman kernel.
	
	\
	
	 In this article, we will study the asymptotics of logarithmic Bergman kernels. And as an application, we will calculate the conditional expectation of density of zeros of Gaussian random holomorphic sections of $H^0(M,L^k)$ which vanishes along a fixed smooth subvariety $V$ of $M$. The logarithmic Bergman kernel behaves like Bergman kernel for the singular metric \cite{SunSun,Punctured, Sun2019} in the sense that it behaves very much like the smooth case ``away" from the subvariety $V$, while it exhibit very different nature ``around" $V$. We use the notation $\epsilon(k)$ to mean a term that is bounded by $Ck^{-r}$ for all $r$, which becomes invisible in any asymptotic expansion in inverse powers of $k$.
	 
	 \
	 
	  ``Away" from $V$, we have
	 \begin{theo}\label{main1}
	 	For $z\in M$, let $r=d(z,V)$ be the distance. Then
	 	when $r\geq \frac{\log k}{\sqrt{k}}$, we have	$$\rho_{k,V}(z)=\rho_k(z)-\epsilon(k).$$
	 	In particular, $\rho_{k,V}(z)$ has the same asymptotic expansion as $\rho_k(z)$
	 \end{theo}
And ``around" $V$, we have 	
our main result
\begin{theo}\label{main2}
	For $z\in M$, let $r=d(z,V)$ be the distance. Then
	\begin{itemize}
		\item[$\bullet$] when $r\leq \frac{\sqrt{\log k}}{\sqrt{2k}}$, we have	$$\frac{\rho_{k,V}}{\rho_k}(z)=(1-e^{-kr^2})(1+R_k(z))$$
		where $|R_k(z)|\leq C_\epsilon\frac{kr^2}{k^{-1/2+\epsilon}}$ for any $\epsilon>0$.
		\item[$\bullet$]When $r$ satisfies
		$$ \frac{\sqrt{\log k}}{\sqrt{2k}}<r< \frac{\log k}{\sqrt{k}}.$$
		Then we have	$$\frac{\rho_{k,V}}{\rho_k}(z)=1-O(e^{-kr^2})$$
	\end{itemize}

\end{theo}
	We would like to comment that like the asymptotic of the Bergman kernel in the smooth case, the asymptotic around $V$ also depend on mainly on the geometry of the manifolds $M$ and $V$. But unlike the smooth case where the geometry only kicks in from the second term, the dependence on the geometry appears from the first term in the asymptotic of the logarithmic Bergman kernel near $V$. When $V$ has singularities, the asymptotic should be more interesting near the singularities.
	Our proof depends on two important tools: the Ohsawa-Takegoshi-Manivel extension theorem and the asymptotic of the off-diagonal Bergman kernel developed by Shiffman-Zelditch etc. 
	
	\
	
	It is necessary to compare our logarithmic Bergman kernel to the partial Bergman kernel studied by Zelditch-Zhou 	\cite{ZelditchZhou-Interface,ZelditchZhou2019}, Ross-Singer\cite{Ross2017} and Coman-Marinescu\cite{ComanMarinescu}, etc. The partial Bergman kernel is for the space of sections that vanishes along a subset to a order than grows with the power $k$, while our logarithmic Bergman kernel vanishes to order 1(or one can study other fixed orders). So it appears to the author that partial Bergman kernel is more ``analytic", while our logarithmic Bergman kernel is more ``algebraic". Also, so far, the author have not been able to find results for partial Bergman kernel that vanishes a long a submanifold of codimension $\geq 1$.
	
	\
	
	Recall that the expectation of density of zeros of sections of a line bundle is defined as a $(1,1)$-current. More precisely, given a line bundle $L\to M$, with $\dim M=m$, and a Hermitian inner product $H$ on $H^0(M,L)$, we have a complex Gaussian measure $d\mu$ on $H^0(M,L)$. Let $Z_s$ denote the zero variety of $s\in H^0(M,L)$. Then the expectation is defined as
$$\E(Z_s)(f)=\int_{H^0(M,L)}d\mu(s)\int_{Z_s}f$$
for any smooth $(m-1,m-1)$-form $f$ with compact support. Given a subset $V$ of $M$, the conditional expectation of density of zeros of sections of $L^k$ is denoted by $\Zb_k(z|V)$, defined by
$$\Zb_k(z|V)(f)=\E_{(H^0(M,L^k),d\mu_k)}(\int_{Z_s}f|\quad s|_V=0), $$
for any smooth $(m-1,m-1)$-form $f$ with compact support, where $d\mu_k$ is the complex Gaussian measure on $H^0(M,L^k)$ corresponding to the Hermitian inner product on this space. The inner product on $\hcal_{k,V}$ is inherited from that on $H^0(M,L^k)$. We denote by $d\gamma_k$ the induced complex Gaussian measure on $\hcal_{k,V}$, then we have the following
$$\E_{(H^0(M,L^k),d\mu_k)}(\int_{Z_s}f|\quad s|_V=0)=\E_{(\hcal_{k,V},d\gamma_k)}(\int_{Z_s}f)$$
Then it follows from \cite{Shiffman1999} and \cite{Sun2019GA} that 
\begin{prop}
	$$\Zb_k(z|V)=\frac{i}{2\pi}\ddbar \log \rho_{k,V}+k\omega.$$
\end{prop}
Therefore, as a corollary of theorem \ref{main1}, we have
\begin{cor}\label{cor1}
	For $z\in M\backslash V$, we have 
	$$\Zb_k(z|V)=k\omega(1+O(\frac{1}{k^2})).$$
	
\end{cor}
And as a corollary of theorem \ref{main2}, we have
\begin{cor}\label{cor2}
		For $z_0\in  V$, we fix a normal coordinates $(w_1,\cdots,w_m)$ so that $V$ is given by $w_{n+1}=\cdots=w_m$ and $\omega(z_0)=\sqrt{-1}\sum_{i=1}^{m}dw_i\wedge d\bar{w}_i$. Then for $z\in \C^m$ have the scaling limit
	$$\Zb_{\infty}(z|V)=\lim_{k\to \infty}\Zb_k(z_0+\frac{z}{\sqrt{k}}|V)=\sqrt{-1}\sum_{i=1}^{n}dz_i\wedge d\bar{z}_i+\sqrt{-1}\ddbar\log (e^{\sum_{i=n+1}^{m}|z_i|^2}-1)$$
\end{cor}
	
	\
	The structure of this article is as follows. We will first do the calculations in the complex projective space, which gives us important insight for the general picture. Then we use the Ohsawa-Takegoshi-Manivel extension theorem to prove lemma \ref{theoR}, which is very important for this article, and theorem \ref{main1}. Then we use the asymptotic of the off-diagonal Bergman kernel to prove theorems \ref{main2Part1} and \ref{main2Part2} which form theorem \ref{main2}. Then we quickly prove corollaries \ref{cor1} and \ref{cor2}.
	
	\
	
	\textbf{Acknowledgements.} The author would like to thank Professor Bernard Shiffman for his continuous and unconditional support. The author would also like to thank Professor Chengjie Yu and Professor Song Sun for many very helpful discussions.

	\section{Complex projective space and its geometry}
	Let $[Z_0,\cdots,Z_N]$ be the homogeneous coordinates of $\CP^N$. $U_0=\{[1,z],z\in \C^N\}$ is a coordinate patch with $z_i=\frac{Z_i}{Z_0}$. The $Z_i$'s can be identified as generating sections in 	$H^0(\CP^N,\ocal(1))$. In particular, $Z_0$ is a local frame in $U_0$. Then on $U_0$, the Fubini-Study form $\omega=\frac{i}{2}\ddbar\log (1+|z|^2)$ has the following explicit form $$\omega=\frac{i}{2}\frac{(1+|z|^2)\sum dz^i\wedge d\bar{z}_i-(\sum \bar{z}_idz_i)(\sum z_id\bar{z}_i)}{(1+|z|^2)^2}$$
	and the point-wise norm of $Z_0$ is given by
	$$|Z_0|^2_{FS}(z)=\frac{1}{1+|z|^2}=e^{-\varphi}$$
	For simplicity, we use the volume form $\frac{\omega^N}{\pi^N}$ instead of $\frac{\omega^N}{N!}$. Then the total volume of $M=\CP^N$ is $1$. With the Riemannian metric associated to $\omega$, the distance between two points $[Z]$ and $[W]$ in $M$ is given by $\arccos \frac{|<Z,W>|}{|Z||W|}$.
	
	\
	
Sections in	$H^0(\CP^N,\ocal(k))$ are represented by homogeneous polynomials of variables $Z_0,$$\cdots,Z_N$. Endowed with the inner product
$$<s_1,s_2>=\int_{\CP^N}s_1\bar{s}_2e^{-k\phi}\omega^N$$
 $H^0(\CP^N,\ocal(k))$ is then a Bergman space, denoted by $\hcal_k$. And the Bergnan kernel $\rho_k(z)$, by $U(N+1)$-invariance, is constant. So we have
 $\rho_k(z)=\dim H^0(\CP^N,\ocal(k))=N_k$. Therefore, we can read out an orthonormal basis of $\hcal_k$ from the binomial expansion of $\rho_k(z)e^{k\varphi}=N_k(1+|z|^2)^k$.

 \

Let $V\subset \CP^N$ be a linear subspace of codimension $m$. But a $U(N+1)$-change of coordinates, we can assume that in the coordinate patch $U_0=\{[1,z],z\in \C^N\}$, $V$ is defined by $z_1=\cdots=z_m=0$. Then the sections of $H^0(\CP^N,\ocal(1))$ that vanish along $V$ are generated by $Z_i$, $i=1,\cdots,m$, where $Z_i$ is represented by $z_i$ in $U_0$. More generally, we consider the space $\hcal_{k,V}$, consisting of sections of $H^0(\CP^N,\ocal(k))$ that vanishes along $V$. By a quick calculation, the Bergman kernel of $\hcal_{k,V}$ on $U_0$ is
$$\rho_{k,V}=N_k\frac{(1+|z|^2)^k-(1+\sum_{i=m+1}^{N}|z_i|^2)^k}{(1+|z|^2)^k}$$
In particular $$\rho_{1,V}=(N+1)\frac{\sum_{i=1}^{m}|z_i|^2}{1+|z|^2}$$
We will need the following lemma:
\begin{lem}\label{projectivedistance}
	For each point $[Z]\in M$, the number $\arcsin\sqrt{\frac{\rho_{1,V}([Z])}{N+1}}$ is just the distance of $[Z]$ to $V$ under the Fubini-Study distance.
\end{lem}
\begin{proof}
	This is straightforward, since the distance of $[1,z]\in U_0$ to $[1,0]$ is just $\arccos\frac{1}{1+|z|^2}$.
\end{proof}
	We clearly have 
	\begin{eqnarray*}
		\lim_{k\to \infty}\sqrt{-1}\ddbar\log [(1+|z|^2/k)^k&-&(1+\sum_{i=m+1}^{N}|z_i|^2/k)^k]\\
		&=&\sqrt{-1}\ddbar\log (e^{|z|^2}-e^{\sum_{i=m+1}^{N}|z_i|^2})\\
			&=&\sqrt{-1}\sum_{i=m+1}^{N}dz_i\wedge d\bar{z}_i+\sqrt{-1}\ddbar\log(e^{\sum_{i=1}^{m}|z_i|^2}-1)
	\end{eqnarray*}
which gives us a hint on why corollary \ref{cor2} should be true.

	\section{Setting-up for the general case}
	Recall that the on-diagonal Bergman kernel can also be defined as
	$$\rho_k(z)=\sup_{\parallel s\parallel=1} |s(z)|^2_h$$
	And at each point $p\in M$, the supremum is achieved by an unique(up to a multiple of $e^{i\theta}$) unit section, denoted by $s_p$, called the peak section at $p$.  $s_p$ can also be characterized as the unit section that is orthogonal to all holomorphic sections that vanishes at $p$.
	 The techniques of peak sections have been very useful in the calculation of asymptotics of Bergman kernel, for example \cite{Lu2000On}. For reader's convenience and for later use, let us quickly copy some details of the construction of the peak sections.
	
	\
	
	First of all, we can choose local holomorphic coordinates $\{w_a\}$ centered at a given point $p$ and local \kahler potential $\phi$ for the \kahler form $\omega$
	so that $$\phi(w)=\sum_aw_a\bar{w}_a+O(w^3)$$
	Then by a careful change of coordinates we can make $\phi$ be of the form
	$$\phi(w)=\sum_a w_a\bar{w}_a+\sum P_{abcd}w_aw_b\bar{w}_c\bar{w}_d+\textit{higher order terms}$$
	Then by the rescaling of coordinates $z_a=\sqrt{k}w_a$, $\phi$ becomes
	$$\Phi(z)=|z|^2+\frac{1}{k}P(z)+k^{-3/2}Q(z)+O(k^{-2})$$
	And the volume form $(i\ddbar \Phi)^m$ is of the form
	$$J=1+k^{-1}p(z)+k^{-3/2}q(z)+O(k^{-2})$$
	Then one choose local frame $\sigma_0$ for the line bundle $L^k$ on the ball $|z|\leq k^{1/4}$ for example so that $|\sigma_0|^2_h=e^{-\Phi}$. We will mention $\sigma_0$ as a "normal frame".
	 One then modify $\sigma_0$ to get the peak section. More precisely, $\sigma_0$, when regarded as a global discontinuous section of $L^k$, extending by zero outside our ball, is almost orthogonal to all holomorphic sections vanishing at $p$. In fact, by the symmetry of $e^{-|z|^2}$, we have $$|<\sigma_0,\tau>|\leq Ck^{-1}\parallel \tau\parallel$$
	and $$\parallel\sigma_0\parallel^2=\pi^m(1+ak^{-1}+O(k^{-2}))$$
	Then by H\"{o}rmander's $L^2$-techniques, we can modify $\sigma_0$ to get a global section, by introducing an error of the size $\epsilon(k)$, due to the fact that $e^{-|z|^2}$ decays very fast near the boundary of our ball.  Therefore, under the local frame $\sigma_0$, the peak section $s_p$ is represented by a holomorphic function of the form $\sqrt{\frac{k^m}{\pi^m}}(1+O(k^{-1}))$. We will be using the important property of $s_p$ that 
	$$|s_p(z)|_h=\epsilon(k),$$
	for $z$ whose distance to $p$ is $\geq \frac{\log k}{\sqrt{k}}$, since $e^{(\log k)^2}=\epsilon(k)$.
	
	\
	
Let $V$ be a subvariety of $M$, then for $k$ large enough, we have the exact sequence
$$0\to H^0(M,L^k\otimes I_V)\to H^0(M,L^k)\to H^0(V,L^k)\to 0$$
where $I_V$ is the ideal sheaf of $V$. So $\hcal_{k,V}=H^0(M,L^k\otimes I_V)$, and
the orthogonal complement $\hcal_{k,V}^{\perp}$ is isomorphic to $H^0(V,L^k)$.
We have the following lemma.
\begin{lem}\label{span}
	$\hcal_{k,V}^{\perp}$ is spanned by the peak sections $s_p$ for $p\in V$.
\end{lem}
\begin{proof}
	We denote by $W$ the linear space spanned by the peak sections $\{s_p\}_{p\in V}$. Then for any $s\in \hcal_{k,V}$, and each $p\in V$, $s$, vanishing at $p$, is orthogonal to $s_p$. Therefore $s_p\in \hcal_{k,V}^{\perp}$, namely $W\subset \hcal_{k,V}^{\perp}$. On the other hand, if $s\in \hcal_k$ is orthogonal to $W$, then $s$ has to vanish at each $p\in V$, so $s\in \hcal_{k,V}$.
\end{proof}
We denote by $\pi$ the restriction map $H^0(M,L^k)\to H^0(V,L^k)$. The restriction of $\pi$ on $\hcal_{k,V}^{\perp}$ is an isomorphism, and is denoted by $R$.

\

 When $V$ is of pure dimension $n$, $H^0(V,L^k)$ is also endowed with a Hermitian inner product by integrating over the smooth part of $V$. We want to show that $\frac{R}{\sqrt{k^r}}$ is a quasi-isometry. For this purpose, we need to use the Ohsawa-Takegoshi-Manivel extension theorem. There are several versions of this theorem, for example, \cite{OhsawaDemailly, Manivel1993,Varolin2015}. The version that is most useful for our purpose is the one from \cite{Ohsawa5}. In order to state the theorem, we copy the setting-up from \cite{Ohsawa5}.

 \

 Let $M$ be a complex manifold of dimension $m$ with continuous  measure $d\mu_M$ and let $(E,h) $ be a holomorphic Hermitian vector bundle over $M$. Let $S$ be a closed complex submanifold of dimension $n$.
Consider a class of continuous functions $\Psi:M\to [-\infty,0)$ such that
\begin{itemize}
	\item[(1)] $S\subset\Psi^{-1}(-\infty)$
	\item[(2)] If $S$ is $n$-dimensional around a point $x$, there exists a local coordinate $(z_1,\cdots,z_m)$ on a neighborhood $U$ of $x$ such that $z_{n+1}=\cdots=z_m=0$ on $S\cap U$ and
	$$\sup_{U\backslash S}|\Psi(z)-(m-n)\log \sum_{n+1}^{m}|z_j|^2|<\infty$$
\end{itemize}
The set of such functions $\Psi$ will be denoted by $\sharp(S)$. Clearly, the condition (2) does not depend on the choice of local coordinate. For each $\Psi\in \sharp(S)$, one can associate a positive measure $dV_M[\Psi]$ on $S$ as the minimum element of the partially ordered set of positive measure $d\mu$ satisfying
$$\int_{S}fd\mu\geq \overline{\lim_{t\to \infty}}\frac{2(m-n)}{\sigma_{2m-2n-1}}\int_Mfe^{-\Psi} \chi_{R(\Psi,t)}dV_M$$
for any nonnegative continuous function $f$ with supp$f\subset\subset M$
. Here $\sigma_m$ denotes the volume of the unit sphere in $\R^{m+1}$, and $\chi_{R(\Psi,t)}$ denotes the characteristic function of the set
$$R(\Psi,t)=\{x\in M|-t-1<\Psi(x)<-t \}.$$
Let $\Theta_h$ be the curvature form of the fiber metric $h$. Let $\Delta_h(S)$ be the set of functions $\tilde{\Psi}$ in $\sharp(S)$ such that, for any point $x\in M$, $e^{-\tilde{\Psi}}h=e^{-\hat{\Psi}}\hat{h}$ around $x$ for some plurisubharmonic function $\hat{\Psi}$ and some fiber metric $\hat{h}$ whose curvature form is semipositive in the sense of Nakano.
\begin{theo}[Ohsawa-Takegoshi-Manivel,\cite{Ohsawa5}]
Let $M$ be a complex manifold with a continuous volume form $dV_M$, let $E$ be a holomorphic vector bundle over $M$ with a $C^{\infty}$ fiber metric $h$, let $S$ be a closed complex submanifold of $M$, let $\Psi\in \sharp(S)$ and let $K_M$ be the canonical line bundle of $M$. If the following are satisfied,
\begin{itemize}
	\item[1)] There exists a closed subset $X\subset M$ such that \begin{itemize}
		\item[(a)] $X$ is locally negligible with respect to $L^2$ holomorphic functions, i.e., for any local coordinate neighborhood $U\subset M$ and for any $L^2$ holomorphic function $f$ on $U\backslash X$, there exists a holomorphic function $\tilde{f}$ on $U$ such that $\tilde{f}|_{U\backslash X}=f$.
		\item[(b)] $M\backslash X$ is a Stein manifold with intersects with every component of $S$.
	\end{itemize}
\item[2)] $\Theta_h\geq 0$ in the sense of Nakano.
\item[3)]$(1+\delta)\Psi\in \Delta_h(S)\cap C^{\infty}(M\backslash S)$ for some $\delta>0$.

\end{itemize}
then there exists a constant $C$ such that for any $f\in H^0(S,E\otimes K_M|_S)$ such that
$$\int_S |f|^2_{h\otimes (dV_M)^{-1}}dV_M[\Psi]<\infty,$$
there exists $F\in H^0(M,E\otimes K_M)$ such that
$$\int_M|F|^2_{h\otimes (dV_M)^{-1}}dV_M\leq (C+\delta^{-3/2})^2\int_S |f|^2_{h\otimes (dV_M)^{-1}}dV_M[\Psi].$$
If $\Psi$ is plurisubharmonic, the constant $(C+\delta^{-3/2})^2$ can be chosen to be less than $256\pi $.

\end{theo}
In our situation, $V$ is the $S$ in the theorem. The volume form is $dV_M=\frac{\omega^m}{m!}$. The vector bundle $E$ is the line bundle $L^k-K_M$ with a twisting of the metric $e^{-k\phi}\otimes dV_M=e^{-k\phi+\kappa}$. Since our manifold $M$ is projective, the $X$ in the theorem exists. Let $N_V$ denote the normal bundle of $V$, with the metric induced by $\omega$ on $T_M$. Let $r$ denote the length of vectors in $N_V$. Denote by $N_V(\rho)$ the subset of vectors with length $r<\rho$, the for $\rho$ small enough, the exponential map $$\exp: N_V(\rho)\to M$$
is a diffeomorphism of $N_V(\rho)$ with its image. Then $r$ is a function in a neighborhood of $V$ in $M$, we then choose a nonnegative smooth function $\chi$ on $[0,\infty)$, which is concave and satisfies the following conditions
\begin{itemize}
\item[(1)]$\chi(x)=x$ for $x\leq \frac{(\log k)^2}{k}$;
\item[(2)]$\chi(x)$ is constant for $x\geq \frac{(10\log k)^2}{k}$
\end{itemize}
So $\chi(r^2)$ can be seen as a smooth function on $M$, which is constant away from $V$. Then we twist the metric on $L^k-K_M$ by $e^{\beta_k \chi(r^2)}$, for $\beta_k$ to be determined.
Then we let $\Psi=(m-n)\log r^2$ and extend it smoothly to be defined on $M$. Clearly this function $\Psi$ satisfies the two conditions in the definition of $\sharp(V)$. Also $(1+1)\Psi\in \Delta_h(V)\cap C^{\infty}(M\backslash S)$ for $h=e^{-k\phi+\kappa+\beta_k \chi(r^2)}$ when $\beta_k$ not too big. We want to make $\beta_k$ as large as possible. So we calculate
$$\ddbar \chi(r^2)=\chi'\ddbar r^2+\chi^{''}\partial r^2\wedge\bar{\partial}r^2$$
By our construction of $\chi$, we have $0\leq \chi'\leq 1$ and $\chi^{''}\leq 0$. So one can see that for $k$ large, we can allow $\beta_k$ to be of the size $k-O(\frac{\log k}{\sqrt{k}})$. Finally, we calculate the measure $dV_M[\Psi]$. By integrating along fibers, one sees easily that for our $\Psi$, the measure $dV_M[\Psi]$ is just the smooth measure $\frac{\omega^n}{n!}$. 

\

Now we can apply the Ohsawa-Takegoshi-Manivel extension theorem to get that for any $f\in H^0(V,L^k)$, one can find $F\in H^0(M,L^k)$ satisfying the inequality:
\begin{equation}
\int_M |F|^2e^{-k\phi+\beta_k\chi(r^2)}\frac{\omega^m}{m!}\leq C\int_V |f|^2e^{-k\phi}\frac{\omega^n}{n!}
\end{equation}
for some constant $C$ independent of $k$. 
 For simplicity, we assume that $\int_V |f|^2e^{-k\phi}\frac{\omega^n}{n!}=1$.

\

First of all, since $e^{-\beta_k\chi(r^2)}=\epsilon(k)$ for $r\geq \rho=\frac{\log k}{\sqrt{k}}$,
 we see that
$$\int_{M\backslash N_V(\rho)}|F|^2e^{-k\phi}dV_{M,\omega}=\epsilon(k)$$
So we can see that the mass of $F$ is concentrated within a small tubular neiborhood of $V$ with radius $\rho=\frac{\log k}{\sqrt{k}}$. To estimate the integral within the small neighborhood. We only need to notice that
$$\int_{\C}|z|^{2a}e^{-k|z|^2}dV=\frac{\pi a!}{k^{a+1}}$$
Since $\beta_k=k-O(\frac{\log k}{\sqrt{k}})$, by integrating along fibers of $N_V$ first, one can see that
$$\int_{N_V(\rho)}|F|^2e^{-k\phi}dV_{M,\omega}\approx \frac{1}{k^{m-n}} \int_{N_V(\rho)} |F|^2e^{-k\phi+\beta_k\chi(r^2)}\frac{\omega^m}{m!}$$
Therefore, we have the following
\begin{theo}\label{theoR}
The restriction map $R: \hcal_{k,V}^{\perp}\to H^0(V,L^k)$ has norm satisfying
$$\parallel R\parallel^2=O(\frac{1}{k^{m-n}})$$
\end{theo}
\begin{rem}
	This theorem is very important for our subsequent arguments. If one wishes to be more precise than the big O, one need to have optimal constant in the Ohsawa-Takegoshi-Manivel extension theorem \cite{ZhouZhu1,ZhouZhu2}.
\end{rem}
As a direct application, we have theorem \ref{main1}:
\begin{theo}
For $z\in M$, let $r=d(z,V)$ be the distance. Then
when $r\geq \frac{\log k}{\sqrt{k}}$, we have	$$\rho_{k,V}(z)=\rho_k(z)-\epsilon(k).$$
In particular, $\rho_{k,V}(z)$ has the same asymptotic expansion as $\rho_k(z)$
\end{theo}
\begin{proof}
Consider the peak section $s_z$. We know that for any point $w$ with distance $d(z,w)\geq \frac{\log k}{\sqrt{k}}$, the length
$|s_z(w)|_h=\epsilon(k)$. So, the $L^2$ norm of $s_z|_V$ is $\epsilon(k)$. So by the the theorem above, if we write $s_z=s_1+s_2$ with
$s_1\in \hcal_{k,V}$ and $s_2\in \hcal_{k,V}^{\perp}$, then $\parallel s_2\parallel^2=\epsilon(k)$. Therefore $\parallel s_1\parallel^2=1-\epsilon(k)$ and 
$$|s_1(z)|_h=|<s_1,s_z>||s_z(z)|_h=|<s_1,s_1>||s_z(z)|_h=(1-\epsilon(k))|s_z(z)|_h$$
So $\rho_{k,V}(z)\geq \rho_k(z)-\epsilon(k)$, on the other hand, we clearly have $\rho_{k,V}(z)< \rho_k(z)$, so the theorem is proved.
\end{proof}

\section{Calculations near $V$}

Next, we study the asymptotic of $\rho_{k,V}$ around $V$. For this,
we need the off-diagonal asymptotics for the Bergman kernel, about which we now recall some details from \cite{sz1,sz2,sz3}.

Let $\pi:X\to M$ be the circle bundle of unit vectors in the dual bundle $L^*\to M$ with respect to $h$. Sections of $L^k$ lift to equivariant functions on $X$. Then $s\in H^0(M,L^k)$ lifts to a CR holomorphic function on $X$ satisfying $\hat{s}(e^{i\theta }x)=e^{ik\theta}\hat{s}(x)$. We denote the space of such functions by $\hcal_k^2(X)$. The Szeg\"{o} projector  is the orthogonal projector
$$\Pi_k:L^2(X)\to \hcal_k^2(X),$$
which is given by the Szeg\"{o} kernel(Bergman kernel)
$$\Pi_k(x,y)=\sum \hat{s}_j(x)\overline{\hat{s}_j(y)} \quad (x,y\in X). $$
(Here the functions $\hat{s}_j$ are the lifts to $\hcal_k^2(X)$ of the orthonormal sections $\{s_j\}$; they provide an orthonormal basis for $\hcal_k^2(X)$.)

The covariant derivative $\nabla s$ of a section $s$ lifts to the horizontal derivative $\nabla_h\hat{s}$ of its equivariant lift $\hat{s}$ to $X$; the horizontal derivative is of the form
$$\nabla_h\hat{s}=\sum_{j=1}^{m}(\frac{\partial \hat{s}}{\partial z_j}-A_j\frac{\partial \hat{s}}{\partial \theta}dz_j)$$
For $z=\pi(x),w=\pi(y)\in M$, we will write
$$|\Pi_k(z,w)|=|\Pi_k(x,y)|,$$
in particular, on the diagonal $\Pi_k(z,z)>0$ is the same as our previous notation $\rho_k(z)$.  For each point $z_0\in M$, we choose a neighborhood $U$, a local coordinate chart $\rho: (U,z_0)\to (\C^m,0)$, and a preferred local frame at $z_0$, which is a local frame $e_L$ such that
$$\parallel e_L(z)\parallel_h=1-\frac{1}{2}|\rho(z)|^2+\cdots$$
For $u=(u_1,\cdots,u_m)\in \rho(U)$, $\theta\in (-\pi, \pi)$, let
$$\tilde{\rho}(u_1,\cdots,u_m,\theta)=\frac{e^{i\theta}}{|e_L^*(\rho^{-1}(u))|_h}e_L^*(\rho^{-1}(u))\in X$$
so that $(u_1,\cdots,u_m,\theta)\in \C^m\times \R$ give local coordinates on $X$. We then write
$$\Pi_k^{z_0}(u,\theta;v,\varphi)=\Pi_k(\tilde{\rho}(u,\theta),\tilde{\rho}(v,\varphi))$$

\begin{theo}[\cite{sz1,sz2,sz3}]
	Let $(L,h)$ be a positive Hermitian holomorphic line bundle over a compact m-dimensional \kahler manifold $M$. $\omega=\frac{i}{2}\Theta_h$ is the \kahler form. Let $z_0\in M$, then
	\begin{itemize}
		\item[(i)]\begin{eqnarray*}
			&&\pi^m k^{-m}\Pi^{z_0}_k(\frac{u}{\sqrt{k}},\frac{\theta}{\sqrt{k}};\frac{v}{\sqrt{k}},\frac{\phi}{\sqrt{k}})\\
			&=&e^{i(\theta-\phi)+u\cdot v-\frac{1}{2}(|u|^2+|v|^2)}[1+\sum_{r=1}^{l}k^{-\frac{r}{2}}p_r(u,v)+k^{-\frac{l+1}{2}}R_{kl(u,v)}],
		\end{eqnarray*}
		where the $p_r$ are polynomials in $(u,v)$ of degree $\leq 5r$(of the same parity as $r$), and
	$$|\nabla^jR_{kl}(u,v)|\leq C_{jl\epsilon b}k^{\epsilon}\quad \textit{for}\quad |u|+|v|<b\sqrt{\log k},$$
	for $\epsilon, b\in \R^+, j,l\geq 0$. Furthermore, the constant $C_{jl\epsilon b}$ can be chosen independent of $z_0$.
	\item[(ii)] For $b>\sqrt{j+2l+2m}, j, l\geq 0$, we have
	$$|\nabla_h^j\Pi_k(z,w)|=O(k^{-l})$$
	uniformly for $\textit{dist}(z,w)\geq b\sqrt{\frac{\log k}{k}}$.
	\end{itemize}
\end{theo}
The so called normalized Bergman kernel $P_k(v,z_0)$ was also defined in \cite{sz1} as
$$P_k(v,z_0)=\frac{|\Pi_k(v,z_0)|}{\sqrt{\Pi_k(v,v)}\sqrt{\Pi_k(z_0,z_0)}}$$
which contains plenty of information of the geometry of the image of $M$ under the Kodaira embedding\cite{sun1}. Recall that it was proved in \cite{sz1} the following estimations.
\begin{theo}
	$$P_k(z_0+\frac{u}{\sqrt{k}},z_0+\frac{v}{\sqrt{k}})=e^{-\frac{1}{2}|u-v|^2}(1+R_N(u,v))$$
	where the remainder satisfies the following \begin{eqnarray*}
		|R_k(u,v)|&\leq& \frac{C_1}{2}|u-v|^2k^{-1/2+\epsilon}\\
		|\nabla R_k(u,v)|&\leq& C_1|u-v|k^{-1/2+\epsilon}\\
		|\nabla^jR_k(u,v)|&\leq& C_jk^{-1/2+\epsilon}\\
	\end{eqnarray*}
for $|u|+|v|<b\sqrt{\log k}$, where the $C_i's$ all depend on $b$.
\end{theo}

Locally, the orthonormal basis $\{s_i\}$ are represented by holomorphic functions $\{f_i\}$, so the Kodaira embedding $\Phi_K$ is given by
$$\Phi_k(z)=[f_0(z),\cdots,f_{N}(z)]$$
We denote by $$Q_k(z,w)=\sum f_i(z)\bar{f}_i(w)$$
So we have $$Q_k(z,z)=\frac{N^m}{\pi^m}(1+O(\frac{1}{k}))e^{k\phi}$$
For simplicity, we first work on the case of codimension 1 to illustrate the idea of calculations. One immediately realizes that this idea works for codimension $\geq 1$. Now we assume $V$ is a smooth divisor $D$. 
Notice that the property required for the choice of $e_L$ by the theorem above is only that
$\phi=|z|^2+$higher order terms. So we are allowed to choose local coordinates $z$ so that $D$ is defined by $\{z_m=0\}$.

We use the notation $f_{i,m}=\frac{\partial}{\partial z_m}f_i$.
Then by taking the derivatives $\frac{\partial}{\partial \bar{z}_m}P(z,z)$ and $\frac{\partial^2}{\partial z_m\partial \bar{z}_m}P(z,z)$, and since $\partial \varphi(z_0)=0$ and $\ddbar \varphi(z_0)=\sum dz_i\wedge d\bar{z}_i$, we get the following estimations
\begin{eqnarray}
	\sum f_i(z_0)\overline{f_{i,m}(z_0)}&=&O(k^{m-1})\\
	\sum f_{i,m}(z_0)\overline{f_{i,m}(z_0)}&=&\frac{k^{m+1}}{\pi^m}(1+O(k^{-1}))
\end{eqnarray}
We denote by $f=(f_0,\cdots,f_N)$ and $f_{,m}=(f_{1,m},\cdots,f_{N,m})$. Then we define an unit section of $L^k$ as
$$\alpha_{z_0}(z)=\frac{\sum \overline{f_{i,m}(z_0)}s_i(z)}{|f_{,m}(z_0)|}$$
The estimations above implies that
$$|\alpha_{z_0}(z_0)|_h^2=O(k^{m-3})$$
We need also to estimate the norm of $\alpha_{z_0}(z)$ for points $z\neq z_0$. For $z,w\in U$, by the asymptotics for the off-diagonal Bergman kernel, we have the following
$$\sum f_i(z/\sqrt{k})\overline{f_i(w/\sqrt{k})}=\frac{k^m}{\pi_m}e^{z\cdot w}[1+\sum_{r=1}^{l}k^{-r/2}p_r(z,w)+k^{-\frac{l+1}{2}}R_{kl}(z,w)]$$
where $p_r$ and $R_{kl}$ are slightly different from those in the theorem, but enjoy similar estimations. Therefore, we have
$$\sum f_i(z)\overline{f_{i,m}(z_0)}=O(k^{m})$$
for $z\in D$ satisfying $|z|<\sqrt{2m+3}\frac{\sqrt{\log k}}{\sqrt{k}}$, since $z_m=0$. This implies that
$$|\alpha_{z_0}(z)|_h^2=O(k^{m-1})e^{-k|z|^2},$$
for these $z$.
Finally, when $\textit{dist}(z,z_0)\geq \sqrt{2m+3}\sqrt{\frac{\log k}{k}}$, we can use part (ii) of Shiffman-Zelditch's theorem to get that
$$|\alpha_{z_0}(z)|_h^2=O(k^{-1})$$
Therefore we can estimate the $L^2$ norm of $\alpha_{z_0}$ on $D$ to get
$$\int_D |\alpha_{z_0}(z)|_h^2\frac{\omega^{m-1}}{(m-1)!}=O(1+\frac{1}{k})=O(1)$$
as $\int_{\C^{m-1}}^{}e^{-k|z|^2}dV=O(k^{1-m})$. So we have proved the following
\begin{lem}\label{almostortho}
	$\alpha_{z_0}$ is almost orthogonal to the space $\hcal_{k,D}^{\perp}$. More precisely, if we write $$\alpha_{z_0}=s_1+s_2,$$
	with $s_1\in \hcal_{k,D}$ and $s_2\in \hcal_{k,D}^{\perp}$, then $\parallel s_2\parallel^2=O(k^{-1})$.
\end{lem}
Before proceeding, we want to explain the meaning of this lemma in the sense of complex projective geometry. A complex vector space equipped with a Hermitian inner product can be identified with its dual space by a conjugate linear map. In our case, let $W=H^0(M,L^k)$, then the Kodaira map $\Phi_k$ maps $M$ to $\PP W^*$. Fixing an orthonormal basis $(s_0,\cdots,s_N)$, and local frame $e_L$, then $\Phi_k(z)=(f_0(z),\cdots,f_N(z)$, where $s_i=f_ie_L^k$. Then by taking the complex conjugate $(\overline{f_0(z)},\cdots,\overline{f_N(z)})$, we get a conjugate-holomorphic embedding $\overline{\Phi_k}:M\to \PP W$, with $\overline{\Phi_k}(z)=[\sum \overline{f_i(z)}s_i]$. What interesting is that $\overline{\Phi_k}(p)$ is just the complex line in $W$ containing the peak section $s_p$. And lemma \ref{span} implies that $\overline{\Phi_k}(D)$ linearly span $\PP \hcal_{k,D}^{\perp}$. So the preceding lemma, in this setting, says that the image of $\frac{\partial}{\partial z_m}(z_0)$ under the tangent map of $\overline{\Phi_k}$ is almost orthogonal to the linear subspace  $\PP \hcal_{k,D}^{\perp}$.

We can apply similar calculation for the point in the $z_m$-disk passing through $z_0$, which have coordinates $(0,\cdots,0,z_m)$ with $z_m$ small. We use $v$ to denote the points on this disk. We again define
$$\alpha_v=\frac{\sum \overline{f_{i,m}(v)}s_i}{|f_{,m}(v)|}$$
as an unit section of $L^k$. Then we estimate the point-wise norm of $\alpha_v$ along $D$ by differentiating the function $Q_k(z,w)$. We have
\begin{eqnarray}
	\sum f_i(z)\overline{f_{i,m}(v)}&=&O(k^{m})|e^{kv\cdot z}|=O(k^{m})\\
	\sum f_{i,m}(v)\overline{f_{i,m}(v)}&=&\frac{k^{m+1}}{\pi^m}(1+O(k^{-1}))(1+k|v|^2)e^{k|v|^2}
\end{eqnarray}
for $z\in D$ satisfying $|z|<\sqrt{2m+3}\frac{\sqrt{\log k}}{\sqrt{k}}$. Therefore, for these $z$, we have
$$|\alpha_v(z)|^2_h=O(k^{m-1})e^{-k|z|^2-k|z|^2}$$
So the integral over this patch of $z$ is $O(1)$ as $\alpha_{z_0}$.
And for the remaining $z\in D$, we still have $|\alpha_v(z)|^2_h=O(1/k)$, so in conclusion, we have
$$\int_D |\alpha_v|_h^2\frac{\omega^{m-1}}{(m-1)!}=O(1)$$
Again, this implies that $\alpha_v$ is almost orthogonal to $\hcal_{k,D}^{\perp}$ with the same estimation as in lemma \ref{almostortho}. Now we put all these ideas together to get
\begin{theo}We have the following estimation for the logarithmic Bergman kernel along the disk:
	 $$1-\frac{P_k^2(v,z_0)}{1-\beta(r)^2}\leq \frac{\rho_{k,D}}{\rho_k}(v)\leq 1-P_k^2(v,z_0),$$
	 where $r=|v|=|v_m|$ and $\beta(r)=C\int_{0}^{r}\sqrt{1+kx^2}e^{kx^2/2}dx$ with $C$ independent of $k$.
\end{theo}
\begin{proof}
We decompose $\overline{f}(v)-\bar{f}(z_0)=b+c$, with $b\in \hcal_{k,D}$ and $c\in \hcal_{k,D}^{\perp}$. Let $d$ denote the distance in the Fubini-Study metric. Then
$$\cos d( \bar{f}(v),\bar{f}(z_0))=\frac{|< \bar{f}(v),\bar{f}(z_0)>|}{|f(v)||f(z_0)|}=\frac{<\bar{f}(z_0)+c,\bar{f}(z_0)>}{|f(v)||f(z_0)|}$$
$$\cos d(\bar{f}(v),\hcal_{k,D}^{\perp})=\frac{|< \bar{f}(v),\bar{f}(z_0)+c>|}{|f(v)||f(z_0)+c|}=\frac{|\bar{f}(z_0)+c|}{|f(v)|}$$
So $$\frac{\cos d( \bar{f}(v),\bar{f}(z_0))}{\cos d(\bar{f}(v),\hcal_{k,D}^{\perp})}=\frac{<\bar{f}(z_0)+c,\bar{f}(z_0)>}{|\bar{f}(z_0)+c||f(z_0)|}$$
 Then
	$$|c(r)|\leq \int_0^r O(\frac{1}{\sqrt{k}})\sqrt{\frac{k^{m+1}}{\pi^m}(1+kx^2)}e^{kx^2/2}dx $$
	So $$\frac{|c(r)|}{|f(z_0)|}\leq C\int_{0}^{r}\sqrt{1+kx^2}e^{kx^2/2}dx=\beta(r)$$
	When $\frac{|c(r)|}{|f(z_0)|}<1$, we have
	$$1\leq \frac{\cos d(\bar{f}(v),\hcal_{k,D}^{\perp})}{\cos d( \bar{f}(v),\bar{f}(z_0))}\leq (1-(\frac{|c(r)|}{|f(z_0)|})^2)^{-1}$$
	By lemma \ref{projectivedistance}, we are interested in $\lambda_v=\sin^2 d(\bar{f}(v),\hcal_{k,D}^{\perp})=\frac{\rho_{k,D}}{\rho_k}$, which satisfies
	$$1-(1-(\frac{|c(r)|}{|f(z_0)|})^2)^{-1}\cos^2 d( \bar{f}(v),\bar{f}(z_0))\leq \lambda_v\leq 1-\cos^2 d( \bar{f}(v),\bar{f}(z_0))$$
	 We have $\frac{|< \bar{f}(v),\bar{f}(z_0)>|}{|f(v)||f(z_0)|}=\frac{|\Pi_k(v,z_0)|}{\sqrt{\Pi_k(v,v)}\sqrt{\Pi_k(z_0,z_0)}}$ is just the normalized Bergman kernel $P_k(v,z_0)$.
	 So

\end{proof}

	Notice that the term $\beta(r)$ is small  when $r$ is small. With the substitution $R=\sqrt{k}r$, it becomes $$\frac{C}{\sqrt{k}}\int_{0}^{R}\sqrt{1+x^2}e^{x^2/2}dx$$
	So for fixed $R$, $\beta=O(\frac{1}{\sqrt{k}})$.

More generally, when $V$ is a smooth subvariety, we can choose local coordinates $(z_1,\cdots,z_m)$ so that $V$ is defined as $z_m=z_{m-1}\cdots=z_{m-r+1}=0$. Then we can repeat our calculations for the divisor case without any difficulties. More precisely, in each normal direction at a point of $V$, we can apply an unitary change of coordinates, so that that direction is contained in the space spanned by $\frac{\partial}{\partial z_m}$
So the conclusions for the divisor case hold for this more general case.

Notice that when $|z|$ is small, $|z|$ is about the distance of $z$ to $V$, since $\omega(z)=\sum \delta_{ij}dz_i\wedge d\bar{z}_i+O(|z|)$. More precisely, we have
$$d(z,D)=|z|(1+O(|z|)),$$ so we can talk about the asymptotics without going local. In particular, when $|z|\leq \frac{\log k}{\sqrt{k}}$,
$$e^{-k|z|^2+kd^2(z,V)}=(1+k|z|^2\frac{\log k}{\sqrt{k}})$$
When $r\leq \frac{\sqrt{\log k}}{\sqrt{2k}}$, we have
$$\beta^2(r)=O(\frac{\log k}{\sqrt{k}}),$$ so we have the following theorem.
\begin{theo}\label{main2Part1}
	For $z\in M$, let $r=d(z,V)$ be the distance. Then
when $r\leq \frac{\sqrt{\log k}}{\sqrt{2k}}$, we have	$$\frac{\rho_{k,V}}{\rho_k}(z)=(1-e^{-kr^2})(1+R_k(z))$$
where $|R_k(z)|\leq C_\epsilon\frac{kr^2}{k^{-1/2+\epsilon}}$ for any $\epsilon>0$.
\end{theo}
\begin{rem}
	If the reader is careful enough, he/she must have noticed there is a gap area between our estimations around $V$ and away from $V$, namely when $$ \frac{\sqrt{\log k}}{\sqrt{2k}}<r< \frac{\log k}{\sqrt{k}},$$
	which, very interestingly, have also been seen in the case of Bergman kernel for Poincar\'{e} type metrics \cite{SunSun,Punctured, Sun2019}, where it was called the "neck". Luckily, unlike the Poincar\'{e} type metrics, the "neck" in our case is not very difficult.
\end{rem}

\

Let $z_0\in M$ with distance $r$ satisfying
$$ \frac{\sqrt{\log k}}{\sqrt{2k}}<r<\frac{\log k}{\sqrt{k}}.$$
We consider the peak section $s_{z_0}$. Using the normal coordinates centered at $z_0$, and the normal frame, we know that $$|s_{z_0}(z)|_h^2=(1+O(\frac{1}{k}))\frac{k^m}{\pi^m}e^{-k|z|^2}$$
Let $w\in V$ be the point that is closest to $z_0$. Then for $z\in V$ in this coordinates patch, we have
$$|z|^2\approx |w|^2+|z-w|^2,$$
since we are looking at geodesics of very small scales, meaning things work like Euclidean spaces.
Therefore
$$\int_V |s_{z_0}(z)|_h^2\frac{\omega^{m-r}}{(m-r)!}=O(k^r)e^{-k|r|^2}$$
Therefore, the image of the orthogonal projection of $s_{z_0}$ onto $\hcal_{k,V}^{\perp}$ has $L^2$ norm $O(e^{-kr^2})$. If we decompose $s_{z_0}=s_1+s_2$, with $s_1\in \hcal_{k,V}$ and $s_2\in \hcal_{k,V}^{\perp}$, then $\parallel s_1\parallel^2=1-O(e^{-kr^2})$ and
$|s_1(z_0)|_h=|<s_1,s_{z_0}>||s_{z_0}|_h=<s_1,s_1>|s_{z_0}|_h=(1-O(e^{-kr^2}))|s_{z_0}|_h$. Therefore
$$\rho_{k,V}(z_0)>|s_1(z_0)|_h^2=(1-O(e^{-kr^2}))\rho_k(z_0)$$

So we have proved:
\begin{theo}\label{main2Part2}
	Let $z_0\in M$ with distance to $V$ $r$ satisfies
	$$ \frac{\sqrt{\log k}}{\sqrt{2k}}<r< \frac{\log k}{\sqrt{k}}.$$
	 Then we have	$$\frac{\rho_{k,V}}{\rho_k}(z)=1-O(e^{-kr^2})$$
\end{theo}

\

\begin{proof}[Proof of corollary \ref{cor1}]
	We have $\rho_{k,V}=\frac{k^m}{\pi^m}(1+O(\frac{1}{k}))$, so $i\ddbar \log \rho_{k,V}=O(\frac{1}{k})\omega$.
\end{proof}

\begin{proof}[Proof of corollary \ref{cor2}]
	
	In this normal coordinates, let $r=d(w,V)$, then we have $r^2=\sum_{i=n+1}^{m}|w_i|^2(1+o(\frac{1}{k^{1/4}}))$ for $w$ satisfying $|w|^2\leq \frac{\log k}{k}$, by Toponogov's compparison theorem. We also choose local frame so that $\phi=|w|^2+\textit{higher order terms}$.
	 Therefore by theorem \ref{main2Part2}, we have
	$$(\rho_{k,V}e^{k\phi})(\frac{z}{\sqrt{k}})=\frac{k^m}{\pi^m}(e^{|z|^2}-e^{\sum_{i=1}^{n}|z_i|^2})(1+o(\frac{1}{k^{1/4}}))$$
	Then by taking the limit of $\sqrt{-1}\ddbar\log (\rho_{k,V}e^{k\phi})(\frac{z}{\sqrt{k}})$, we get the conclusion.
\end{proof}
	\bibliographystyle{plain}

	\bibliography{references}
	
\end{document}